\documentclass{amsart}
\usepackage{amsmath}
\usepackage{amssymb}
\usepackage{enumerate}
\usepackage{graphicx}
\usepackage{bm}
\usepackage[numbers]{natbib}
\usepackage{xcolor}
\usepackage{url}

\newcommand{\real}{\mathbb{R}}

\newcommand{\vn}{\bm{n}}
\newcommand{\vx}{\bm{x}}
\newcommand{\vu}{\bm{u}}
\newcommand{\vv}{\bm{v}}
\newcommand{\vf}{\bm{f}}

\newcommand{\vg}{\bm{g}}
\newcommand{\vz}{\bm{z}}
\newcommand{\vM}{\bm{M}}
\newcommand{\vA}{\bm{A}}
\newcommand{\vB}{\bm{B}}
\newcommand{\vD}{\bm{D}}
\newcommand{\vE}{\bm{E}}
\newcommand{\vK}{\bm{K}}
\newcommand{\vL}{\bm{L}}
\newcommand{\vP}{\bm{P}}
\newcommand{\vI}{\bm{I}}
\newcommand{\vQ}{\bm{Q}}
\newcommand{\vS}{\bm{S}}
\newcommand{\vone}{\bm{1}}
\newcommand{\vzero}{\bm{0}}

\newcommand{\vgamma}{\bm{\gamma}}
\newcommand{\vLambda}{\bm{\Lambda}}

\newcommand{\vq}{\bm{q}}

\newcommand{\cG}{\mathcal{G}}

\newcommand{\true}{\text{true}}
\newcommand{\avg}{\text{avg}}
\DeclareMathOperator{\diag}{diag}
\DeclareMathOperator*{\argmin}{argmin}

\newcommand{\Mb}[1]{\left[{#1}\right]}
\newcommand{\M}[1]{\left({#1}\right)}

\newtheorem{theorem}{Theorem}[section]
\newtheorem{lemma}{Lemma}[section]
\newtheorem{assumption}{Assumption}[section]
\newtheorem{proposition}{Proposition}[section]
\newtheorem{remark}{Remark}[section]

\author[L. Borcea \and F. Guevara~Vasquez \and A.V. Mamonov]
{Liliana Borcea \and Fernando Guevara~Vasquez \and Alexander V.
Mamonov}
\address{Mathematics, U. of Michigan, 2074 E Hall, 530 Church St, Ann
Arbor, MI 48109-1043}
\email{borcea@umich.edu}

\address{Mathematics, U. of Utah, 155 S 1400 E RM 233, Salt Lake City, UT 84112-0090}
\email{fguevara@math.utah.edu}

\address{Mathematics, U. of Houston, 4800 Calhoun Rd. Houston, TX, 77004}
\email{avmamonov@uh.edu}

\title[Discrete Liouville identity]{A discrete Liouville identity for numerical reconstruction of 
Schr\"odinger potentials}

\begin{document}

\begin{abstract} 
We propose a discrete approach for solving an inverse problem for
Schr\"odinger's equation in two dimensions, where the unknown potential
is to be determined from boundary measurements of the Dirichlet to
Neumann map. For absorptive potentials, and in the continuum, it is
known that by using the Liouville identity we obtain an inverse
conductivity problem.  Its discrete analogue is to find a resistor
network that matches the measurements, and is well understood.  Here we
show how to use a discrete Liouville identity to transform its solution
to that of  Schr\"odinger's problem.  The discrete Schr\"odinger
potential given by the discrete Liouville identity can be used to
reconstruct the potential in the continuum in two ways.  First, we can
obtain  a direct but coarse reconstruction by interpreting the values of
the discrete Schr\"odinger potential as averages of the continuum
Schr\"odinger potential on a special sensitivity grid.  Second, the
discrete Schr\"odinger potential may be used to reformulate the
conventional nonlinear output least squares optimization formulation of the
inverse Schr\"odinger problem. Instead of minimizing the boundary
measurement misfit, we  minimize the misfit between the discrete
Schr\"odinger potentials. This results in a better behaved optimization
problem that converges in a single Gauss-Newton iteration, and gives
good quality reconstructions of the potential, as illustrated by the
numerical results.
\end{abstract} 

\maketitle

\section{Introduction}
\label{sec:intro}
Consider the  boundary value problem 
\begin{equation}
 \begin{aligned}
 L_{\sigma,q} v \equiv -\nabla \cdot [ \sigma \nabla v ] + q v &= 0 
~\text{in}~\Omega,\\
 v & = f ~\text{on}~ \partial\Omega,
 \end{aligned}
 \label{eq:sq}
\end{equation}
in a  simply connected, bounded domain $\Omega\subset\real^2$ with $C^2$ 
boundary $\partial\Omega$, and  $f \in H^{1/2}(\partial\Omega)$. 
The positive and bounded coefficient $\sigma(\vx)$ is called the
conductivity and $q(\vx)$ is the Schr\"odinger potential.
They are the unknowns in inversion, to be determined from  the Dirichlet
to Neumann (DtN) map
$\Lambda_{\sigma,q} : H^{1/2}(\partial\Omega) \to
H^{-1/2}(\partial\Omega)$ defined by 
\begin{equation}
\Lambda_{\sigma,q} f = \sigma \vn \cdot
\nabla v |_{\partial\Omega},
\end{equation}
where  $\vn$ is the unit outer normal at $\partial \Omega$ and $v$
solves \eqref{eq:sq}.
The case $q=0$ is known as
the inverse conductivity or electrical impedance tomography problem. 
The case $\sigma=1$ is the inverse Schr\"odinger problem.

Our goal in this paper is to introduce a novel method, based on parametric model reduction, for the numerical reconstruction 
of the solution of the inverse Schr\"odinger problem ($\sigma = 1$) 
in the {\em absorptive} case $q \geq 0$.  Parametric model  reduction is mainly used for  approximating efficiently 
the response of dynamical systems for design, optimal control and uncertainty 
quantification \cite{benner2015survey}.  We are interested in parametric model reduction  for improving the 
inversion process. This is a largely unexplored direction, but recent progress has been made in 
\cite{borcea2014model,druskin2013solution} for parabolic equations, in \cite{mamonov2015nonlinear}
for the wave equation, in \cite{Borcea:2008:EIT,Borcea:2012:RNA,Borcea:2010:PRN,borcea2010circular} for the inverse 
conductivity problem and in \cite{borcea2005continuum} for a related
inverse spectral problem. The construction of the 
reduced models varies between problems because it must respect the
underlying physics. For example, 
the projection-based reduced models in \cite{mamonov2015nonlinear}  approximate the wave propagator and use causality, whereas the 
projection-based models in \cite{borcea2014model,druskin2013solution} for parabolic equations are obtained by rational 
approximation of the transfer function, the Laplace transform with respect to time of the measurement map. The 
parametric reduced models for the inverse conductivity problem are not projection-based. They are resistor networks 
constructed from very accurate approximations of the DtN map and the
resistances in the network play the role of the parameters in the
parametric model reduction. These
resistor network reduced models rely on the graph theory developed in 
\cite{Colin:1994:REP, Colin:1996:REP, Curtis:1989:CRN, Curtis:1994:FCC, Curtis:1990:DRN, 
Curtis:1991:DNM, Curtis:2000:IPE, Curtis:1998:CPG,
Ingerman:2000:Discrete}.

The advantage of using parametric reduced models for inversion is that  they can lead to iterative algorithms that 
perform much better than the usual least squares data fit methods.  They converge in 
one or two iterations and give quantitatively superior reconstructions. The disadvantage is that it is difficult to 
construct good  parametric reduced models. These must  retain the structure of the governing partial differential equation
so we can extract the unknown parameters from them, 
and capture correctly important phenomena such as the decreased sensitivity of the measurements to 
changes in the parameters away from the boundary in inverse elliptic and parabolic problems.

In this paper we show how to use the resistor networks with circular planar graphs, the reduced 
models for the inverse conductivity problem in \cite{Borcea:2008:EIT,Borcea:2010:PRN,borcea2010circular,Borcea:2012:RNA},  
to solve the inverse Schr\"odinger problem with absorptive potentials in two dimensions. 
The inverse conductivity and Schr\"odinger problems are connected in the continuum by a well known Liouville identity \cite{Sylvester:1987:GUT}. Here we 
show how to connect them in the discrete (network) setting. This is difficult because conductivities 
are defined on edges of the graph of the network, and the potential is associated with the nodes.  The parametric reduced models (networks) 
obtained  as in \cite{Borcea:2012:RNA}, which can be interpreted as five point stencil difference  schemes for Schr\"odinger's 
equation, encode information about the unknown potential $q$, but this is not  restricted to the diagonal of the finite 
difference operator, as expected.  Thus, it is not straightforward to obtain a reconstruction of $q$ from the network. 

Discrete Liouville identities have been proposed before
in \cite{Ingerman:2012:SE} and in \cite{Arauz:2014:DRM,Arauz:2015:OPB} for  the inverse
Schr\"odinger problem on networks.  However, these studies are not concerned with the 
connection with the continuum problem, and in fact it is not clear if the discrete Schr\"odinger problem considered in
\cite{Arauz:2014:DRM,Arauz:2015:OPB} is consistent with 
measurements of the DtN map in the continuum. 

Here we derive a discrete generalized Liouville identity for solving 
the continuum inverse Schr\"odinger  problem with networks. The networks,  with graph $\cG$,  are parametric reduced models that approximate the DtN map $\Lambda_{1,q}$.  
The Liouville identity is defined  on the line graph  $\widetilde \cG$ of  $\cG$, which establishes an isomorphism between the edges of $\cG$ where the conductivities lie and the nodes of 
$\widetilde \cG$ that support the Schr\"odinger potential. We use the Liouville identity to formulate a preconditioned 
Gauss-Newton inversion algorithm. The  preconditioner is obtained from the parametric reduced model and leads to an 
efficient method, as demonstrated   with numerical simulations.

\subsection{Contents}

We start in \S\ref{sec:conti} by recalling how the Liouville transform  
relates the inverse problem for the 
absorptive Schr\"{o}dinger equation to the inverse conductivity problem. 
The inverse problem for resistor networks stated in 
\S\ref{sec:rnet} is the discrete analogue of the  inverse conductivity 
problem. We can transform it to a  discrete analogue of the inverse Schr\"{o}dinger problem using  the generalized Liouville 
identity  derived in   \S\ref{sec:dlv}. The connection between
the continuum and discrete inverse problems is in \S\ref{sec:connect}.
We show in section \S\ref{sec:data} how  to relate measurements of the DtN map of the  Schr\"odinger equation
to the discrete DtN map of a unique resistor network. In  \S\ref{sec:solving} we use the generalized Liouville 
identity to obtain a discrete Schr\"odinger potential from this resistor network. The discrete potential is used in  \S\ref{sec:dtc}  to obtain a reconstruction of the 
continuum
Schr\"odinger potential. The performance of the inversion method is assessed with numerical simulations in \S\ref{sec:num}. 
We conclude with a summary in
\S\ref{sec:discussion}.

\section{Continuum inverse conductivity and Schr\"odinger problems}
\label{sec:conti}

A well-known relation between the
Schr\"odinger $L_{1,q}$ and conductivity $L_{\sigma,0}$ differential
operators is through the Liouville identity (see e.g. \cite{Sylvester:1987:GUT})
\begin{equation}
  \sigma^{-1/2} \circ L_{\sigma,0} \circ \sigma^{-1/2} = L_{1,q},
  \label{eq:lv}
\end{equation}
where $\sigma>0$, $\sigma \in C^2(\Omega)$ and
\begin{equation}
 q =  \frac{\Delta(\sigma^{1/2})}{\sigma^{1/2}}.
  \label{eq:lv:q}
\end{equation}
In the left hand side of \eqref{eq:lv} we have the composition
of three linear operators. Two of them are the operator $v \to
\sigma^{-1/2} v$ denoted, in an abuse of notation, by
$\sigma^{-1/2}$.  

In linear algebra two 
matrices $\vA$ and $\vB$
are {\em congruent} if there is an invertible 
matrix $\vS$ such that
$\vA = \vS \vB \vS^T$ (see e.g. \cite{Horn:2013:MA}).  Borrowing the
terminology, we say that two linear differential
operators $A$ and $B$ are {\em Liouville congruent} if there is a
positive function $s$ for which $A = s \circ B \circ s$. 
This allows us to restate the Liouville identity as follows: 
The operators $L_{\sigma,0}$ and $L_{1,q}$ are Liouville congruent 
when $q$ is given by \eqref{eq:lv:q}.  

Note that the 
Dirichlet to Neumann maps of the conductivity and Schr\"odinger 
problems satisfy \cite{Sylvester:1987:GUT}
\begin{equation}
 \Lambda_{1,q} = \sigma^{-1/2} \circ \Lambda_{\sigma,0} \circ \sigma^{-1/2} - \frac{1}{2}
  \vn \cdot \nabla(\sigma^{1/2}).
  \label{eq:lv:dtn}
\end{equation}
Thus, when $\sigma$ equals a constant near $\partial \Omega$,  the DtN maps are
also Liouville congruent. When $\sigma$ has a nonzero normal derivative at $\partial \Omega$,  
the DtN maps are congruent up to a diagonal (or multiplication
by a function) operator. 

We assume henceforth that $\sigma|_{\partial \Omega} = 1$.
Equations \eqref{eq:lv}--\eqref{eq:lv:dtn} show that in the continuum setting any inverse conductivity 
problem for sufficiently smooth $\sigma$ can be 
formulated as a Schr\"odinger problem. The converse is true only for potentials $q$ that give a 
positive solution $\sigma$ of  \eqref{eq:lv:q}.
This is the case for  absorptive potentials $q \geq 0$. Indeed, the strong maximum 
principle (see e.g. \cite[\S 6.4, Theorem 4]{Evans:1998:PDE}) guarantees that when $q \ge 0$, 
the solution of 
\begin{equation}
 \begin{aligned}
 -\Delta s + q \, s &= 0,~\text{in $\Omega$},\\
 s &=1,~\text{on $\partial\Omega$},
 \end{aligned}
 \label{eq:maxp}
\end{equation}
satisfies $0 < s \leq 1$. Hence, $\sigma=s^2$ satisfies 
\eqref{eq:lv:q} and the operators $L_{\sigma,0}$ and $L_{1,q}$ are Liouville 
congruent. 

We can take the notion of Liouville congruence
further than the classical relations (\ref{eq:lv})--(\ref{eq:lv:q}).
Instead of relating operators $L_{\sigma,0}$ and  $L_{1,q}$, 
we can also consider operators $L_{\sigma_1,0}$ and 
$L_{\sigma_0,q}$, for two positive conductivities $\sigma_0$ and $\sigma_1$ and a potential $q$. 
It follows by straightforward calculations that 
\begin{equation}
(\sigma_1/\sigma_0)^{-1/2} \circ L_{\sigma_1,0} \circ (\sigma_1/\sigma_0)^{-1/2}= 
L_{\sigma_0,q},
\label{eq:glv}
\end{equation}
with $q$ given by
\begin{equation}
q = \frac{ \nabla \cdot [ \sigma_0 \nabla (\sigma_1/\sigma_0)^{1/2}
]}{(\sigma_1/\sigma_0)^{1/2}}.
\label{eq:glv:q}
\end{equation}
The importance of this generalized Liouville identity becomes clear in the next 
section, where we derive a discrete analogue of (\ref{eq:glv})--(\ref{eq:glv:q}).

\section{Discrete inverse conductivity and Schr\"odinger problems}
\label{sec:rnet}
As stated in the introduction, it is useful to view our study in the context 
of parametric model reduction for inversion. We need reduced models that  keep the underlying 
structure of the differential operators $L_{\sigma,q}$ so that we can obtain reconstructions of 
 $\sigma$ and $q$ from them and, in addition, give better conditioned optimization algorithms 
than the usual output least squares. Such parametric reduced models are known for the inverse conductivity problem
in two dimensions. They are based on the circular planar resistor networks studied in \cite{Curtis:1998:CPG,Ingerman:2000:Discrete} and 
are described briefly below. We refer to \cite{Borcea:2008:EIT,Borcea:2010:PRN,borcea2010circular} and the review \cite{Borcea:2012:RNA} for details on how to use the parametric reduced models to determine a discrete Laplacian 
which is consistent with the measurements of the DtN map in the continuum setting, and to recover the 
unknown conductivity. 

In this section we describe the basic tools for  extending the inversion approach  to the  Schr\"{o}dinger problem. 
We begin in \S\ref{sect:3.1} with the formulation of the discrete analogues of the inverse conductivity and 
Schr\"odinger problems. Then we review briefly in \S\ref{sec:dinvcond} the relevant facts about the resistor networks. 
The line graph introduced in \S\ref{sec:linegraph} and the discrete Liouville identity defined in \S\ref{sec:dlv} 
allow us to use the networks for solving the inverse Schr\"odinger problem.

\subsection{The discrete conductivity and Schr\"odinger operators}
\label{sect:3.1}

{The discrete structure of a resistor network is an} undirected
graph $\cG=(V,E)$, where $V$ is a finite set of vertices (nodes) and the edge 
set $E$ is a subset of $\{ \{i,j\} ~|~ i,j \in V, i\neq j \}$. We denote 
the set of functions from a finite set $X$ to $\real$ by $\real^X$, and 
write $\vf(x)$ for $\vf  \in \real^X$ and $x \in X$. All functions and 
operators related to finite sets are written henceforth in bold to 
distinguish them from their continuum counterparts. 

The discrete gradient on a network is the linear operator 
$\vD : \real^V \to \real^E$ that maps $\vf \in \real^V$ to
\begin{equation}
 (\vD\vf)(\{i,j\}) = \vf(i) - \vf(j),~\text{for}~\{i,j\} \in E.
\end{equation}
A sign needs to be specified for each edge, but as long as this sign
convention is fixed, it does not change the subsequent definitions.

A resistor network is defined by its  graph $\cG=(V,E)$ 
and a positive discrete conductivity function 
$\vgamma \in (0,\infty)^E$. We may also define a discrete Schr\"odinger
potential $\vq \in \real^V$ on the vertices of the graph, and introduce 
the  discrete Schr\"odinger operator $\vL_{\vgamma,\vq} : \real^V \to \real^V$,
which maps potentials $\vu \in \real^V$ to
\begin{equation}
\vL_{\vgamma,\vq} \vu = \vD^*[\vgamma \odot (\vD \vu)] + \vq \odot \vu \in \real^V.
\label{eqn:vLgq}
\end{equation}
Here $\vD^*:\real^E \to \real^V$ is the adjoint of the discrete gradient $\vD$ 
and the product $\vf \odot \vg$ for functions $\vf,\vg \in \real^X$ 
(where $X$ is either $V$ or $E$) is understood componentwise i.e., for $x \in X$, 
$(\vf\odot \vg)(x) = \vf(x) \vg(x)$. Explicitly, for each $i \in V$, \eqref{eqn:vLgq} gives 
\begin{equation}
(\vL_{\vgamma,\vq} \vu)(i) = \sum_{j ~\text{s.t.} \{i,j\} \in E}
\vgamma(\{i,j\}) ( \vu(i) - \vu(j) ) + \vq(i)\vu(i).
\end{equation} 

The special case $\vL_{\vgamma,\vzero}$ is the {\em weighted graph
Laplacian} with weight $\vgamma$ (see e.g.
\cite{Chung:1997:SGT,Curtis:1991:DNM}). This is  a  resistor network. The edge $e$ of the graph is a resistor with
conductance $\vgamma(e)$ and the nodes represent electrical connections.
The absorptive case $\vL_{\vgamma,\vq}$ with $\vq > 0$ is also a resistor network,
with each node $i$ connected to the ground (zero potential) by
a resistor with conductance $\vq(i)$.

To define the DtN map of a resistor network, we 
collect the nodes $V$ in two disjoint sets: the  {\em boundary nodes} in the set $B \subset V$, and 
the {\em interior} nodes $I = V \backslash B$. 
The Dirichlet boundary value problem for the network is to find the potential $\vu \in \real^V$ such that
\begin{equation}
 \label{eq:dir}
 \begin{aligned}
 (\vL_{\vgamma,\vq} \vu )_I = (\vL_{\vgamma,\vq})_{II} \vu_I +
 (\vL_{\vgamma,\vq})_{IB} \vu_B&= \vzero,~\text{and}\\
 \vu_B &= \vf,
 \end{aligned}
\end{equation}
with given $\vf \in \real^B$. Here $\vu_I$ is the restriction of $\vu$
to $I$, and the linear operator $(\vL_{\vgamma,\vq})_{BI}: \real^I \to \real^B$ is
defined by $(\vL_{\vgamma,\vq})_{BI} \vv = (\vL_{\vgamma,\vq} \vu)_B$
where $\vu_I = \vv$ and $\vu_B = \vzero$. The other operators are defined similarly.
\begin{lemma}
When the subgraph of $\cG$ induced by $I$ is {\em
connected}, the Dirichlet problem for $\vq \geq \vzero$ has a unique
solution for any $\vf \in \real^B$.  
\end{lemma}
\begin{proof}
The case $\vq = \vzero$ is proven in
\cite{Curtis:2000:IPE,Colin:1998:SG} using a discrete analogue to the
maximum principle. Here we proceed by first showing that
$(\vL_{\vgamma,\vzero})_{II}$ is positive definite. If the subgraph of
$\cG$ induced by $I$ is connected, we have that the weighted graph
Laplacian on this subgraph, which is denoted by
$\vL_{\vgamma_I,\vzero}$, is a symmetric positive semidefinite matrix
with a one dimensional nullspace spanned by the constant vector $\vone
\in \real^I$. The sub-matrix $(\vL_{\vgamma,\vzero})_{II}$ can be
written
\[
(\vL_{\vgamma,\vzero})_{II} = \vL_{\vgamma_I,\vzero} + \vE,
\] 
where $\vE$ is a diagonal matrix, whose only non-zero entries are
positive and correspond to the interior nodes that have an edge in
common with a boundary node. If we decompose $\vv \in \real^I$ as $\vv =
\vv_\perp + \alpha \vone$, with $\vone \in \real^I$ and $\vv_\perp^T \vone = 0$, it is clear
that $\vv^T (\vL_{\vgamma,\vzero})_{II} \vv = \vv_\perp^T
\vL_{\vgamma_I,\vzero} \vv_\perp + \vv^T \vE \vv$, and therefore
$(\vL_{\vgamma,\vzero})_{II}$ must be positive definite. When 
$\vq\geq\vzero$,  notice that
$
(\vL_{\vgamma,\vq})_{II} = (\vL_{\vgamma,\vzero})_{II} + \diag(\vq_I)
$
is the sum of a positive definite and a positive semidefinite
matrix, and must then be positive definite. This proves the solvability
result for $\vq \geq 0$.
\end{proof}

The Dirichlet to
Neumann (DtN) map $\vLambda_{\vgamma,\vq} : \real^B \to \real^B$ associates to the boundary potentials $\vf \in \real^B$ 
the boundary currents 
\begin{equation}
\label{eq:netbdry}
(\vL_{\vgamma,\vq} \vu)_B = (\vL_{\vgamma,\vq})_{BB} \vu_B +
(\vL_{\vgamma,\vq})_{BI} \vu_I,
\end{equation}
where $\vu$
solves (\ref{eq:dir}).  The linear map $\vLambda_{\vgamma,\vq}$ is symmetric, and in 
the case $\vq = \vzero$, it has a  one dimensional null space spanned 
by the vector of all ones ${\bf 1} \in \mathbb{R}^B$. 
To write it more explicitly, let us simplify notation as $
\vL_{BB} := (\vL_{\vgamma,\vq})_{BB}, $ and similar for the $BI$, $IB$ and $II$ blocks 
of $\vL_{\vgamma,\vq}$. 
By solving for $\vu_I$
in \eqref{eq:dir} and substituting in \eqref{eq:netbdry}, we can write  $\vLambda_{\vgamma,\vq}$ as a Schur complement of $\vL_{\vgamma,\vq}$,
\begin{equation}
 \vLambda_{\vgamma,\vq} = \vL_{BB} - \vL_{BI} \vL_{II}^{-1} \vL_{IB}.
 \label{eq:dtn}
\end{equation}

\subsection{The inverse problem for resistor networks}
\label{sec:dinvcond}
The discrete analogue of the inverse conductivity problem is: Find the  conductivity $\vgamma\in (0,\infty)^E$ from the 
DtN map $\vLambda_{\vgamma,\vzero}$, given the underlying graph $\cG=(V,E)$. In what follows we  refer to the 
solution  as $\vgamma(\vLambda_{\vgamma,\vzero})$.

The discrete inverse problem has been solved for circular planar 
graphs \cite{Curtis:1998:CPG,Curtis:1994:FCC,Colin:1994:REP} i.e.,  when $\cG$ 
can be embedded in the plane with no edge crossings 
so that the boundary nodes are on a circle and the interior nodes inside the circle.
The recoverability result in \cite{Curtis:1998:CPG,Curtis:1994:FCC,Colin:1994:REP} 
can be summarized as follows.

\begin{theorem}
\label{thm:recoverability}
The DtN map for a {\em critical} circular planar resistor
network with positive conductivity determines uniquely the conductivity.
\end{theorem}

\noindent A circular planar graph is said to be {\em critical} if the two
following conditions hold: 
\begin{enumerate}
\item[(i)] any two sets of boundary nodes with
the same number $p$ of elements and lying on disjoint segments of the boundary
circle can be linked with $p$ disjoint paths. 
\item[(ii)] the deletion of any edge 
in the graph breaks condition (i). 
\end{enumerate}
For example, among the family of graphs
$C(m,n)$ illustrated in Figure~\ref{fig:cmn},  with $n$ odd, the graphs are critical 
when $m = (n-1)/2$ i.e., when the number $|E|$ of edges is equal to the
number $n(n-1)/2$ of entries in the strictly upper triangular part\footnote{Since $\vLambda_{\vgamma,\vzero}$ is  
symmetric  with rows summing to zero, its independent entries lie in its strictly upper triangular part.}  of $\vLambda_{\vgamma,\vzero}$.

\begin{figure}
\begin{center}
\begin{tabular}{cc}
 \includegraphics[width=0.3\textwidth]{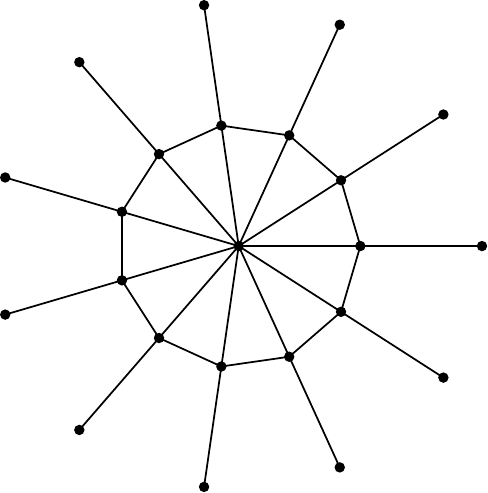} & 
 \includegraphics[width=0.3\textwidth]{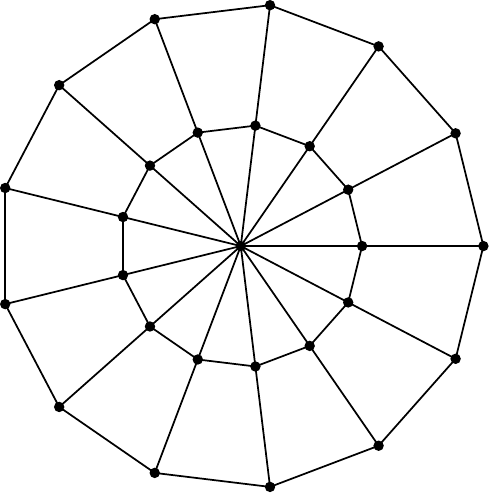}\\
 $C(3,11)$ & $C(4,13)$
\end{tabular}
\end{center}
\caption{Graphs of class $C(m,n)$, where $n$ the number of boundary nodes and $m$ is the number of radial and angular layers.
The boundary nodes are shown as circles and the interior nodes as filled (black) circles.}
\label{fig:cmn}
\end{figure}

\begin{remark} There are other 
families of critical circular planar graphs that allow the number of boundary nodes to be 
even, such as the pyramidal networks in \cite{Borcea:2010:PRN,Curtis:1998:CPG}. In this paper 
we use graphs $C\big(n(n-1)/2,n\big)$ with $n $ odd. 
\end{remark}

There exist direct (optimization-free) reconstruction methods  for networks with graphs  $C((n-1)/2,n)$.
They solve the discrete inverse conductivity problem in a 
finite number of algebraic operations by peeling off sequentially the
layers of the network \cite{Curtis:1994:FCC}.  
Layer peeling and any other methods of reconstructing $\vgamma(\vLambda_{\vgamma,\vzero})$  become 
unstable as the number of edges in the network grows. This is similar to the 
instability of the continuum inverse conductivity problem. Thus, in practice it is impossible 
to approximate the solution of the continuum problem by the discrete conductivity 
of larger and larger networks. However, we can relate the continuum and discrete problems 
by using networks of modest size, as explained in 
\S\ref{sec:connect}.

In \S\ref{sec:conti} we discussed a well-known connection 
between the inverse conductivity and Schr\"odinger problems in the
continuum via the Liouville transform and its generalization 
(\ref{eq:glv})--(\ref{eq:glv:q}). To use the existing results 
on the discrete inverse conductivity problem, we seek to relate the discrete conductivity
$\vgamma$ to a discrete Schr\"odinger potential $\vq$. 
This is not straightforward because the discrete conductivity 
$\vgamma \in \real^E$ is defined on the edges of the graph, while 
$\vq \in \real^V$ is defined at the nodes. To resolve this difficulty,
and to derive the discrete analogue of the generalized Liouville 
transform, we introduce below the line graph and a discrete Schr\"odinger
operator associated with it.

\subsection{The line graph}
\label{sec:linegraph}

To a resistor network with graph $\cG = (V,E)$ we associate a
{\em line graph} $\widetilde{\cG} = (\widetilde{V},\widetilde{E})$.
All quantities and operators associated with the line graph
appear henceforth with tilde. The vertices (nodes) $\widetilde{V}$ and edges
$\widetilde{E}$ of the line graph are defined in terms of $V$ and $E$
of the original graph so that:
\begin{itemize}
\item[{(i)}] there is one vertex of $\widetilde{\cG}$ per edge of $\cG$, 
and so $\widetilde{V}$ is isomorphic to $E$.
\item[{(ii)}] there is an edge between two vertices of $\widetilde{\cG}$ 
if and only if the corresponding edges of $\cG$ share a node.
\end{itemize} 
An example of a graph and its line graph is given in Figure~\ref{fig:linegraph}.

If $\vgamma \in (0,\infty)^E$ is a conductivity on the graph $\cG$, we
define the conductivity $\widetilde{\vgamma} \in (0,\infty)^{\widetilde{E}}$ 
on the associated line graph by geometric averages of $\vgamma$. 
To be more precise, if $e,e' \in E$ are distinct edges in  $\cG$ that share a vertex, then $\{e,e'\} \in \widetilde{E}$ is an 
edge of the line graph $\widetilde{\cG}$ and
\begin{equation}
 \widetilde{\vgamma}(\{e,e'\}) =
 \sqrt{\vgamma(e)\vgamma(e')}.
\end{equation}
The line graph together with the conductivity $\widetilde{\vgamma}$ is 
itself a resistor network. As in section~\ref{sec:rnet} we can define a
discrete gradient $\widetilde{\vD} : \real^{\widetilde{V}} \to
\real^{\widetilde{E}}$, and for some $\widetilde{\vq} \in
\real^{\widetilde{V}}$, we can define a discrete Schr\"odinger operator
$\widetilde{\vL}_{\widetilde{\vgamma},\widetilde{\vq}} :
\real^{\widetilde{V}} \to \real^{\widetilde{V}}$ that maps
$\widetilde{\vu} \in \real^{\widetilde{V}}$ to
\begin{equation}
\label{eq:deflineLgq}
 \widetilde{\vL}_{\widetilde{\vgamma},\widetilde{\vq}} \widetilde{\vu} = \widetilde{\vD}^*[
 \widetilde{\vgamma} \odot (\widetilde{\vD}
 \widetilde{\vu})]
 + \widetilde{\vq} \odot \widetilde{\vu}.
\end{equation}
Since $\widetilde{V}$ is isomorphic to $E$ and 
$\real^{\widetilde{V}}$ is isomorphic to $\real^E$, we can think of
$\widetilde{\vL}_{\widetilde{\vgamma},\widetilde{\vq}}$ as a discrete
Schr\"odinger operator acting on the edges $E$ of $\cG$, with a
Schr\"odinger potential $\widetilde{\vq}$ defined on $E$ as well (i.e., a
function in $[0,\infty)^E$).

\begin{figure}
\begin{center}
 \includegraphics[width=0.4\textwidth]{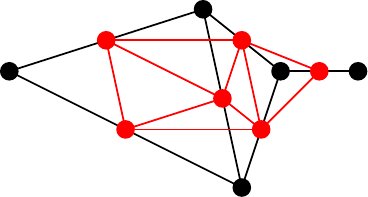}
\end{center}
\caption{A graph $\cG$ in black and its line graph $\widetilde{\cG}$ in red.}
\label{fig:linegraph}
\end{figure}

\subsection{The discrete generalized Liouville identity}
\label{sec:dlv}
With the line graph defined above, we derive here a
discrete analogue to the generalized Liouville identity 
\eqref{eq:glv}--\eqref{eq:glv:q}.
The need for this identity can be explained as follows. Unlike the 
continuum case, where $\sigma \equiv 1$ corresponds to 
$L_{1,0} = \Delta$, the case of constant discrete conductivity 
$\widetilde{\vgamma} = \widetilde{\vone}$ has no special significance in 
the discrete setting. However, the reduced model (resistor network) for the reference conductivity $\sigma_o \equiv 1$ 
(i.e., $q \equiv 0$)
is a good approximation of the Laplacian on the graph $\cG$, as shown in \cite{Borcea:2008:EIT, Borcea:2010:PRN, Borcea:2012:RNA}.
The geometric averages of its edge conductivities $\vgamma_0$ define the line graph Laplacian. We explain in the next section how
to obtain the reduced model (resistor network) for the unknown parameter $\sigma$, related to the unknown $q$. 
The discrete Liouville identity stated in the next theorem allows us to estimate $q$ from this reduced model and the line graph Laplacian.  

\begin{theorem}[Discrete Generalized Liouville Identity]\label{thm:dlv}
Let $\vgamma_0$ and $\vgamma_1$ be two conductivities in $(0,\infty)^E$. Then
their associated weighted graph Laplacians on the line graph satisfy
\begin{equation}
 \widetilde{\vL}_{\widetilde{\vgamma}_1,\vzero}  = \diag(\vgamma_1/\vgamma_0)^{1/2}
 \widetilde{\vL}_{\widetilde{\vgamma}_0,\widetilde{\vq}} \diag(\vgamma_1/\vgamma_0)^{1/2},
 \label{eq:dlv}
\end{equation}
where division and power of vectors is understood componentwise and
\begin{equation} 
 \widetilde{\vq} = - (\widetilde{\vL}_{\widetilde{\vgamma}_0,\vzero}
 [(\vgamma_1/\vgamma_0)^{1/2}]) \odot [(\vgamma_1/\vgamma_0)^{-1/2}].
 \label{eq:dlv:q}
\end{equation}
We say that $\widetilde{\vL}_{\widetilde{\vgamma}_1,\vzero}$ and
$\widetilde{\vL}_{\widetilde{\vgamma}_0,\widetilde{\vq}}$ are {\em Liouville
congruent}, when $\widetilde{\vq}$ is given by \eqref{eq:dlv:q}.
\end{theorem}
\begin{proof}
We obtain from  \eqref{eq:deflineLgq} and the 
 definition of the discrete gradient
$\widetilde{\vD}$ that for any $\vgamma \in (0,\infty)^E$ and $e,f \in E$, $e \ne f$, 
\[
 \widetilde{\vL}_{\widetilde{\vgamma},\vzero}(e,f) = 
 \begin{cases} 
  \widetilde{\vgamma}(\{e,f\}) = -\sqrt{\vgamma(e)\vgamma(f)}, &\text{if
  $\{e,f\} \in \widetilde{E}$,}\\
  0, &\text{otherwise.}
 \end{cases}
\]
Hence for any $\vgamma \in (0,\infty)^E$, and distinct $e,f \in
E$, we have 
\begin{equation}
 \label{eq:one}
 [\diag(\vgamma^{-1/2}) \widetilde{\vL}_{\widetilde{\vgamma},\vzero}
 \diag(\vgamma^{-1/2})](e,f) = 
 \begin{cases} 
  -1, &\text{if
  $\{e,f\} \in \widetilde{E}$,}\\
  0, &\text{otherwise.}
 \end{cases}
\end{equation}
Then the off-diagonal entries of \eqref{eq:dlv} follow from writing
\eqref{eq:one} for $\gamma$ equal to $\vgamma_0$ and $\vgamma_1$,  and equating.

It remains to calculate $\widetilde \vq$, so that the diagonal entries of  \eqref{eq:dlv} are equal.
Because the off-diagonal entries are equal, this is the same as equating 
the sum of the rows
\[
\widetilde{\vL}_{\widetilde{\vgamma}_1,\vzero}\vone = \widetilde{\vL}_{\widetilde{\vgamma}_0,\widetilde \vq} \vone, \qquad 
\vone \in \real^E.
\]
Since
$\widetilde{\vD} \vone =\vzero$, we must also have
$\widetilde{\vL}_{\widetilde{\vgamma},\vzero} \vone = \vzero$, so 
\[
 \vzero =  \widetilde{\vL}_{\widetilde{\vgamma}_0,\widetilde \vq} \vone = [(\vgamma_1/\vgamma_0)^{1/2}] \odot \left(
 \widetilde{\vL}_{\widetilde{\vgamma}_0,\vzero}
 [(\vgamma_1/\vgamma_0)^{1/2}] + \widetilde{\vq} \odot
 [(\vgamma_1/\vgamma_0)^{1/2}] \right).
\]
Solving for $\widetilde{\vq}$ in this equation gives \eqref{eq:dlv:q}.
\end{proof}

\begin{remark}
Theorem \ref{thm:dlv} shows that two weighted graph Laplacians on the line graph, 
$\widetilde{\vL}_{\widetilde{\vgamma}_1,\vzero}$ and $\widetilde{\vL}_{\widetilde{\vgamma}_0,\vzero}$, 
are congruent up to a diagonal term which involves the potential $\widetilde \vq$. Moreover, 
the matrix that carries the congruence relation is diagonal, given by $\diag(\vgamma_1/\vgamma_0)^{1/2}$. 
This is the discrete analogue of the congruence relation \eqref{eq:glv}, and the definition of $\widetilde \vq$ is the 
analogue of \eqref{eq:glv:q}.
\end{remark} 

\section{Connection between the continuum and discrete inverse problems}
\label{sec:connect}
To solve the continuum inverse problem, we 
obtain in \S\ref{sec:data} the networks as parametric reduced models for lumped measurements of the DtN map $\Lambda_{1,q}$. 
These measurements define the DtN map $\Lambda_{\vgamma,0}$ of the network with discrete conductivity $\vgamma$, 
related to the   unknown potential $q$ via  the 
 discrete generalized Liouville 
identity, as described in 
\S\ref{sec:solving}.

\subsection{Resistor networks as parametric reduced models}
\label{sec:data}
The exponential instability of the inverse Schr\"odinger and conductivity problems 
limits the resolution of the reconstructions. In our context this means that the 
size of the reduced models (the networks) cannot be too large, specially for noisy data. 
For the networks with graphs $C\big(n(n-1)/2,n\big)$ considered here,  the number $n$ 
of boundary nodes can be chosen based on the noise level, as explained in Appendix \ref{app:lump}. 
We show here how to define from the measurements of $\Lambda_{1,q}$ a discrete 
DtN map of the network with graph $C\big(n(n-1)/2,n\big)$, our reduced model for the unknown potential $q$ and therefore conductivity. 

In principle, we could just take point measurements of $\Lambda_{1,q}$ at the $n$ boundary nodes. 
However, studies such as \cite{gisser1990electric}  tell us that smooth boundary currents 
(i.e., with Fourier transform supported at small frequency) are better for sensing inhomogeneities 
inside the domain. Thus, we smooth the measurements of $\Lambda_{1,q}$ by lumping them at the $n$ boundary 
points of the network. The lumping is achieved with  $n$ compactly 
supported smooth functions $\phi_1,\ldots,\phi_n$, 
whose disjoint supports are numbered consecutively in $\partial\Omega$,  
in counter-clockwise order. We normalize them by 
$\int_{\partial\Omega} \phi_j = 1$. 

From 
$\Lambda_{1,q}$ we obtain the $n \times n$ 
data matrix $\langle \phi_i , \Lambda_{1,q} \phi_j \rangle$ for $i,j = 1,\ldots,n,$
where $\langle \cdot,\cdot \rangle$ denotes the
$H^{1/2}(\partial\Omega),H^{-1/2}(\partial\Omega)$ duality pair. 
It follows from the relation \eqref{eq:lv:dtn} between the DtN maps 
that for absorptive potentials $q\geq 0$ there is a 
unique conductivity $\sigma$ satisfying \eqref{eq:lv:q} and 
$\sigma|_{\partial\Omega} = 1$, so that
\begin{equation}
 \langle \phi_i , \Lambda_{1,q} \phi_j \rangle
 = \langle \phi_i , \sigma^{-1/2} \circ \Lambda_{\sigma,0}
 \circ \sigma^{-1/2} \phi_j\rangle - \frac{1}{2}\langle \phi_i , [\vn\cdot
 \nabla(\sigma^{1/2})] \phi_j \rangle,
 \label{eq:k}
\end{equation}
for $ i,j= 1,\ldots,n$. Since $\sigma|_{\partial\Omega} = 1$, we can
simplify the first term in \eqref{eq:k}, and the second term vanishes
for $i\neq j$ because the $\phi_i$ have disjoint supports. We then get
\begin{equation}
\langle \phi_i , \Lambda_{1,q} \phi_j \rangle  = \langle \phi_i , \Lambda_{\sigma,0} \phi_j
 \rangle,~\text{for $i,j= 1,\ldots,n$, $i\neq j$},
 \label{eq:knondiag}
\end{equation}
so the off-diagonal entries of the data matrix  are  lumped measurements 
of the DtN map for the associated conductivity problem. We denote this map 
by $\vM(q)$, and define its entries $M_{ij}$ as  
\begin{equation}
 M_{i,j} = \left\{  \begin{array}{ll} \langle \phi_i , \Lambda_{1,q} \phi_j \rangle , \quad & i \ne j, \\
   -\displaystyle\sum_{r=1\ldots n, r\neq i}\langle \phi_i , \Lambda_{1,q} \phi_r \rangle ,  & i = j.
   \end{array}
   \right.
 \label{eq:smeasmat}
\end{equation}
This definition satisfies the conservation of currents relation $\vM
\vone = \vzero$, for $\vone \in \mathbb{R}^B$, 
the vector of all ones at the boundary nodes. Since the measurement
matrix $\vM$ corresponds to taking measurements for a conductivity
problem, we can use Theorem 1 in \cite{Borcea:2008:EIT} to get the following
result.
\begin{theorem}
 \label{prop:compat}
 Let $q \geq 0$  and $\vM(q)$ be the $n\times n$ measurement matrix of
 $\Lambda_{1,q}$ obtained as in \eqref{eq:smeasmat}. There is a
 unique  reduced model, which is the resistor network with DtN map equal to $\vM(q)$, graph 
 $C\big(n(n-1)/2,n\big)$,  and discrete conductivity $\vgamma\big(\vM(q)\big)$.
\end{theorem}

\subsection{From resistor networks to 
Schr\"odinger potentials}
\label{sec:solving}
We have now shown that there is a unique  parametric reduced model 
(resistor network) with DtN map $\vM(\vq)$ and graph $C(n(n-1)/2,n)$.
It remains to determine the Schr\"odinger potential from its discrete 
conductivity $\vgamma\big(\vM(\vq)\big)$.  

A naive approach to reconstructing the potential would be to first obtain a continuum conductivity 
 $\sigma$ from $\vgamma(\vM(q))$ e.g., using the method
based on resistor networks from \cite{Borcea:2008:EIT}, and then 
use the continuum Liouville identity \eqref{eq:lv:q} to get $q$ from 
$\sigma$. This does not work because  $\sigma$ is never exact and 
errors are amplified when taking the Laplacian in \eqref{eq:lv:q}. To illustrate this, we show in 
figure~\ref{fig:twoq} that the conductivities corresponding to two very 
different Schr\"odinger potentials are hard to tell apart.

\begin{figure}
 \includegraphics[width=0.7\textwidth]{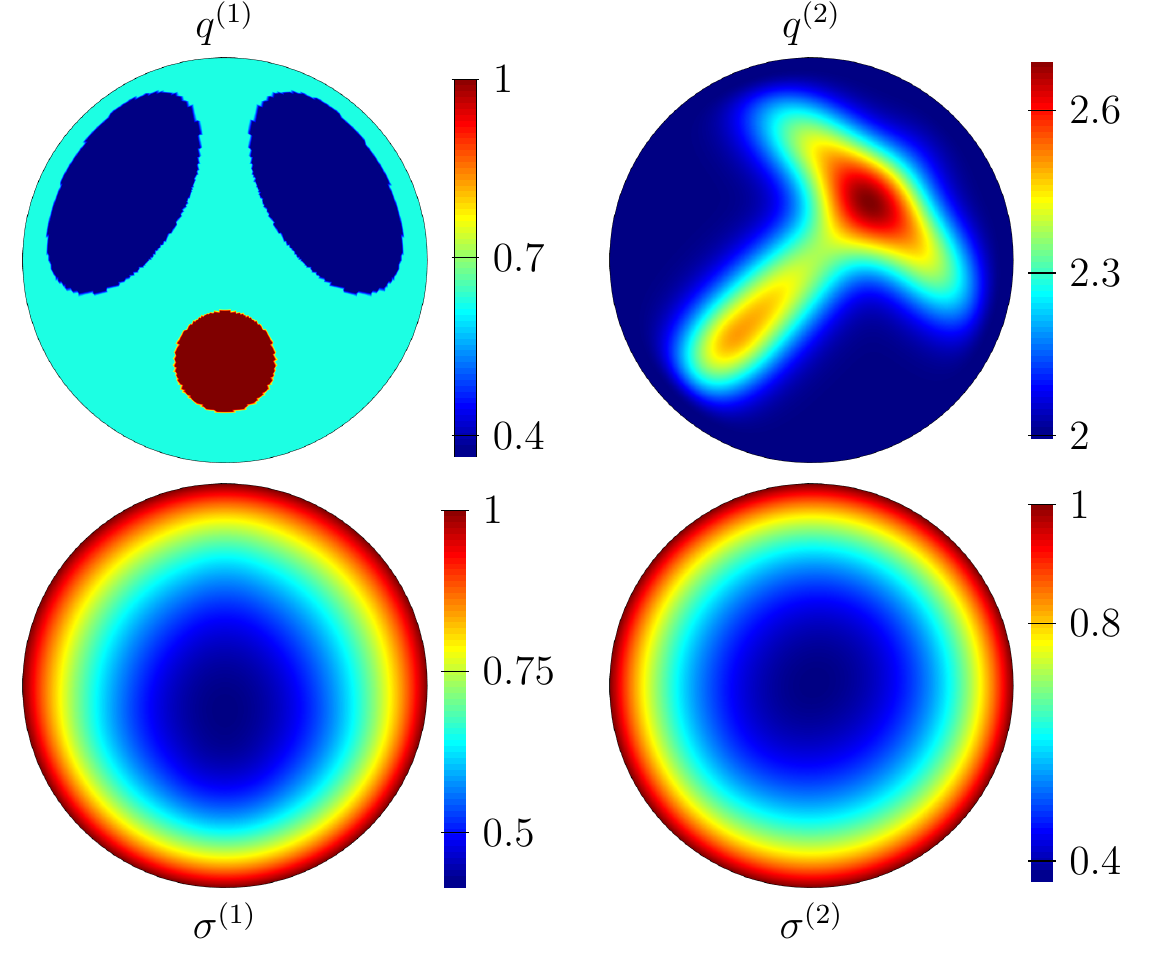}
 \caption{Conductivities $\sigma^{(i)}$, corresponding
 to different Schr\"odinger potentials $q^{(i)}$, $i=1,2$.}
 \label{fig:twoq}
\end{figure}

Our inversion algorithm takes a different route. Its outline is given below, and all the steps are described in detail in the 
next section.
\begin{enumerate}
\item Choose the number $n$ of boundary nodes of the graph $C(n(n-1)/2,n)$. 
\item Obtain  $\vM(q)$ from the given $\Lambda_{1,q}$ and also $\vM(0)$ from 
the calculated $\Lambda_{1,0}$. 
\item Find discrete conductivities $\vgamma_1 \equiv \vgamma(\vM(q))$ and 
$\vgamma_0 \equiv \vgamma(\vM(0))$. 
\item Estimate the average $q_{\avg}$ of the unknown 
potential $q$ from  $\vM(q)$. 
\item On the line graph associated to  $C(n(n-1)/2,n)$
compute the geometric averages $\widetilde{\vgamma}_0$ from 
$\vgamma(\vM(0))$ and form the discrete Laplacian 
$\widetilde{\vL}_{\widetilde{\vgamma}_0,\vzero}$.
\item Use the discrete generalized Liouville identity identity 
\eqref{eq:dlv}--\eqref{eq:dlv:q} with $\vgamma_1$, $\vgamma_0$ 
from step (3) and $\widetilde{\vL}_{\widetilde{\vgamma}_0,\vzero}$ 
from step (5) to obtain the discrete potential
$\widetilde{\vq}$.
\item Use the discrete  potential $\widetilde{\vq}$ and 
the estimated average $q_{\avg}$ to reconstruct the unknown Schr\"odinger 
potential $q$. 
\end{enumerate} 

The use of $\vgamma_0$ at step (3) and of $q_{\avg}$ at step (7)  is a calibration
so that the reconstruction is exact for $q \equiv 0$ and $q \equiv q_{\avg}$. 
In the context of parametric model reduction, we work with three reduced models:
one for the unknown $q$ which we wish to determine, one for the zero potential and one 
for the constant $q_{\avg}$ potential, the average of the unknown $q$. We ask that the reduced 
models share the same discrete Laplacian, calculated for $q = 0$.  This discrete 
Laplacian is a finite difference scheme with steps defined by $\widetilde \vgamma_0$. 
The average potential at step (7) is used as a scaling. We motivate this
scaling by noticing that when $\Omega$ is the unit disk and $q$ is
radially symmetric (i.e. $q \equiv q(r)$), the problem of finding $q$ from
$\vM(q)$ is related to the inverse spectral problem of finding $q$ from the spectrum of
the operator $-\Delta + q$ with Dirichlet boundary conditions
\cite{Borcea:2012:RNA}. To solve the inverse spectral Schr\"odinger problem
it is essential to estimate $q_{\avg}$ accurately as it corresponds to a shift in the
spectrum (see e.g. \cite{Chadan:1997:IIS,borcea2005continuum}). In our case it is
difficult to relate $\vM(q)$ to the spectrum of $-\Delta + q$, so we use
step (7) to ensure that the reconstruction is correct for the constant
$q_{\avg}$.
If the parametric reduced models
share the discrete Laplacian i.e., they give accurate approximations of $\vM(q)$ for a set of functions $q$ 
on the grid defined by $\vgamma_0$, we expect that they account for   the 
asymptotes of the spectrum.

\section{From discrete Schr\"odinger potentials to continuum ones}
\label{sec:dtc}
A numerical estimate of  the solution of the inverse Schr\"odinger problem
is conventionally obtained with the optimization (output least-squares) formulation
\begin{equation}
 \min_{q \geq 0} \frac{1}{2} \| \vK(q) - \vK(q_{\true}) \|_F^2,
 \label{eq:ols}
\end{equation}
where $\| \cdot \|_F$ is the Frobenius norm and $\vK(q)$ denote measurements of $\Lambda_{1,q}$. 
This formulation requires regularization,  based on prior information about $q$ that may not be available. The method is 
widely used but  has a high computational cost, tends to get stuck in local minima
and gives solutions that are biased by the priors.

Instead of minimizing the misfit in the data, we apply a non-linear 
mapping $\vQ: \real^{n\times n} \to \real^{n(n-1)/2}$ that acts as 
an approximate inverse of $\vM(q)$, our lumped measurements of $\Lambda_{1,q}$.  Thus, we solve 
instead the optimization problem
\begin{equation}
 \min_{q \geq 0} \frac{1}{2} \| \vQ(\vM(q)) -
 \vQ(\vM(q_{\true}))
 \|_F^2,
 \label{eq:pols}
\end{equation}
with the Gauss-Newton iteration
\begin{equation}
\begin{aligned}
 \mbox{for}~&k=0,1,\ldots\\
 &q_{k+1} = q_k - D\vM[q_k]^\dagger
 D\vQ[\vM(q_k)]^\dagger( \vQ(\vM(q_k)) -
 \vQ(\vM(q_{\true})) ).
\end{aligned}
\label{eq:gn}
\end{equation}
Here $\dagger$ denotes the
Moore-Penrose pseudoinverse \cite{Golub:2013:MC} and 
$D\vM[q] : L^2(\Omega) \to \real^{n \times n}$ and 
$D\vQ[\vM] : \real^{n \times n} \to \real^{n(n-1)/2}$ are the 
Jacobians of 
the mappings $\vM(q)$ and $\vQ(\vM)$.  Following the approach in \cite{Borcea:2008:EIT} for the inverse conductivity 
problem, we define the nonlinear preconditioner $\vQ(\vM)$ by solving a
discrete Schr\"odinger inverse problem with data $\vM$
(\S\ref{sec:rec}). The initial guess  $q_0$ is obtained by a linear
interpolation of the entries of $\vQ(\vM(q_{\true}))$ on a carefully
chosen grid \S\ref{sec:init}. This ensures that the Gauss-Newton method converges
quickly, usually in one step.

\subsection{The nonlinear preconditioner}
\label{sec:rec}
The mapping $\vQ(\vM)$ is defined by solving a discrete inverse
Schr\"odinger problem with data $\vM$, and ensuring that 
\[
 \vQ(\vM(0)) = \vzero ~ \text{and} ~ \vQ(\vM(q_{\avg})) = q_{\avg} \vone,
\]
with $q_\avg$ estimated as in \S\ref{sec:avgq}.
 The
computation of $\vQ$ involves the following steps:
\begin{enumerate}
 \item Find the parametric reduced models for the zero potential, the constant potential $q_\avg$ and 
 the unknown $q$. That is, determine the discrete conductivities $\vgamma_0
 \equiv \vgamma(\vM(0))$, 
 $\vgamma_{\avg} \equiv \vgamma(\vM(q_{\avg}))$, and
 $\vgamma(\vM)$.
  \item Use the discrete Liouville identity \eqref{eq:dlv:q} with
 reference conductivity $\vgamma_0$ to find the discrete
 potentials $\widetilde{\vq}(\vgamma(\vM))$ and
 $\widetilde{\vq}(\vgamma_{\avg})$ from 
  $\vgamma(\vM)$ and $\vgamma_{\avg}$.

 \item The map $\vQ$ is defined by 
 \begin{equation}
 \label{eq:defQM}
  \vQ(\vM) \equiv q_{\avg}
  \frac{\widetilde{\vq}(\vgamma(\vM))}{\widetilde{\vq}(\vgamma_{\avg})},
 \end{equation}
where the division is understood componentwise. 
\end{enumerate}
The scaling at step (3) ensures that the map $\vQ(\vM)$ gives exact reconstruction 
of the constant potential $q = q_{\avg}$, under the following assumption:
\begin{assumption}
\label{assume1}
All entries of $\widetilde{\vq}({\vgamma_\avg})$ are nonzero.
\end{assumption} 
\noindent  
It is not clear how to prove that this holds, but the assumption has  been verified numerically for many 
graphs and values of $q_{\avg}$ on the unit disk.

\subsection{The sensitivity functions}
\label{sec:sensitivity}

A good preconditioning map $\vQ$  means that $\vQ(\vM)$ is  some approximation of the identity.
We illustrate this numerically in figure~\ref{fig:dqdq} where
we show the sensitivity or Fr\'echet derivative of some of the entries
of the vector $\vQ(\vM(q))$ to changes in $q$. These are  
the ``rows'' of the Jacobian of the preconditioned mapping, 
$D\vQ[\vM(q)] D\vM[q]$. 

The sensitivity  functions are
evidently well localized, so to first order approximation the entries
$\vQ(\vM(q))$ are local averages of $q$. However $D\vQ[\vM(q)]$ is not a
preconditioner in the usual sense because it has a non-trivial
nullspace, as stated in the next proposition. This is handled by taking its pseudoinverse  in the Gauss-Newton iteration 
\eqref{eq:gn}.

\begin{figure}
 \begin{tabular}{c@{\hspace{1em}}c@{\hspace{1em}}c}
  \includegraphics[width=0.24\textwidth]{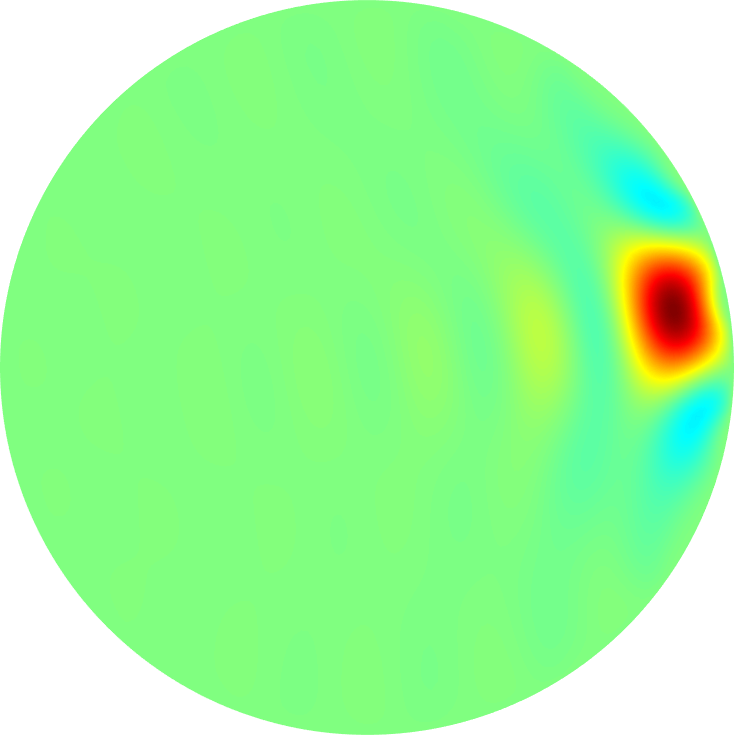} &
  \includegraphics[width=0.24\textwidth]{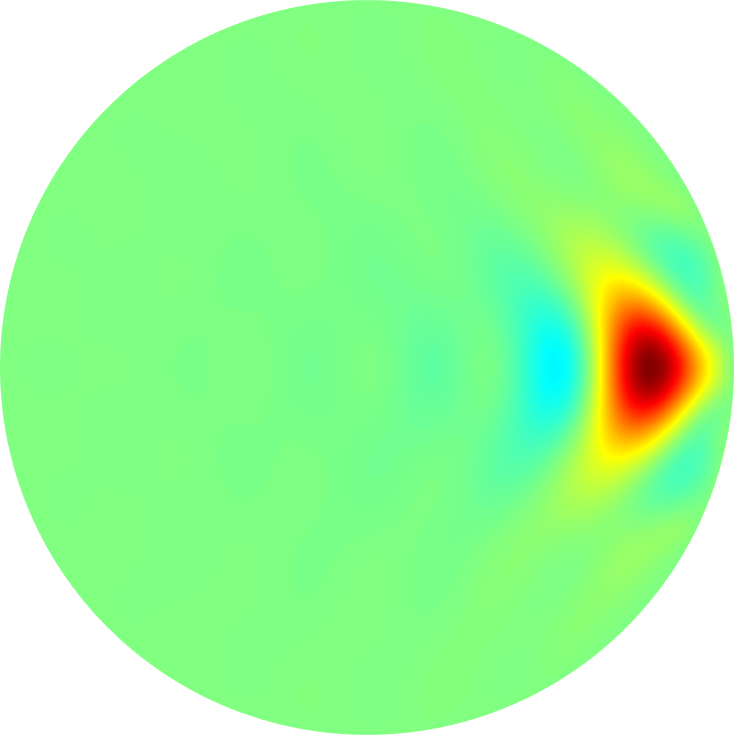} &
  \includegraphics[width=0.24\textwidth]{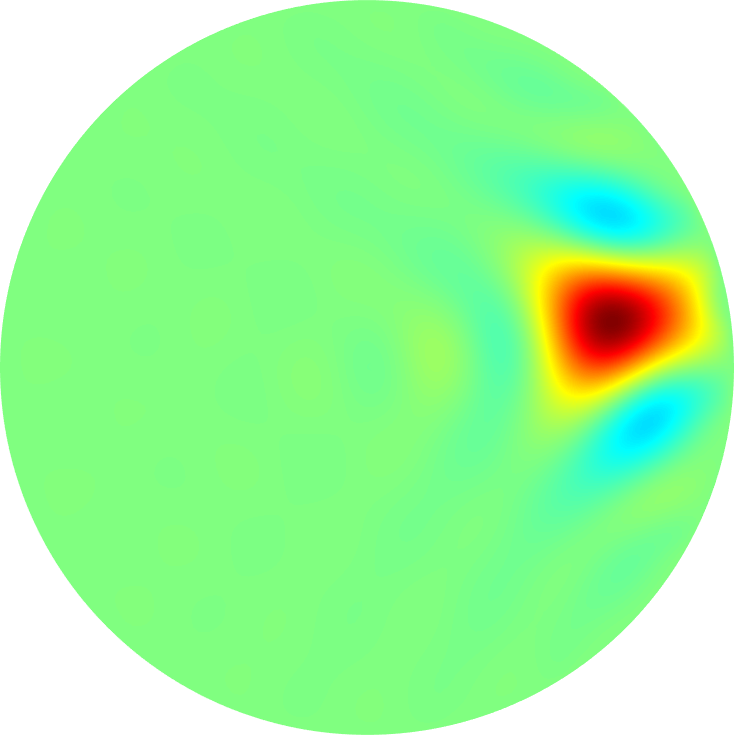} \\
  \includegraphics[width=0.24\textwidth]{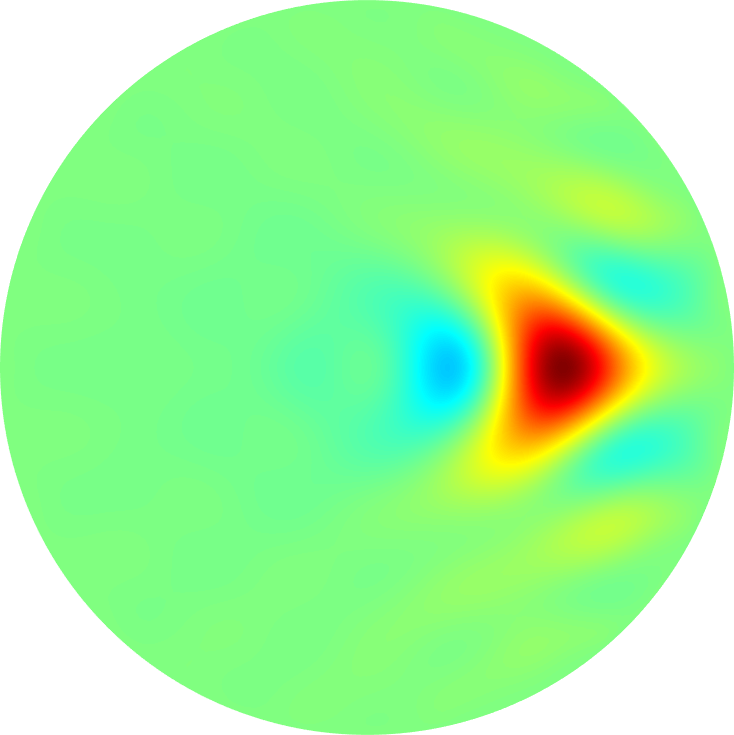} &
  \includegraphics[width=0.24\textwidth]{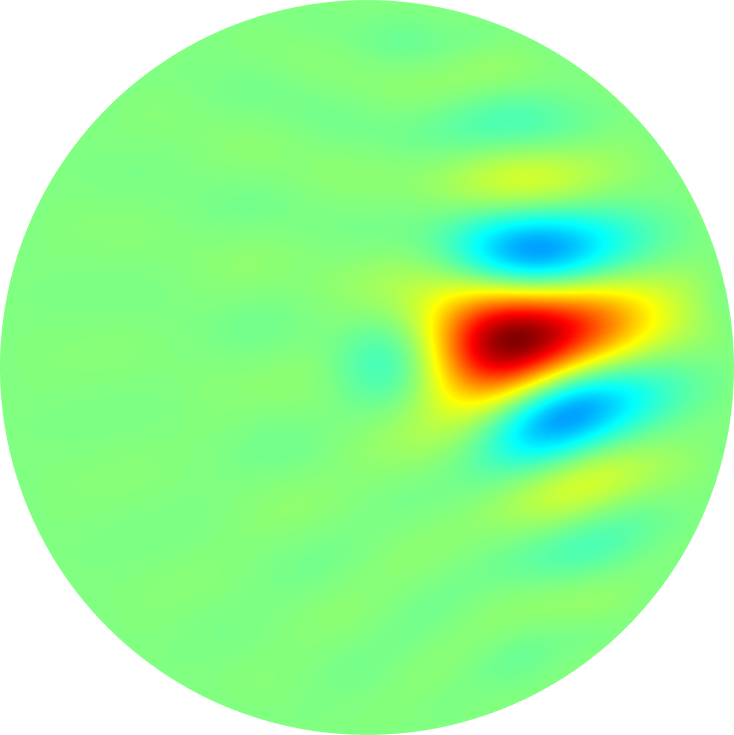} &
  \includegraphics[width=0.24\textwidth]{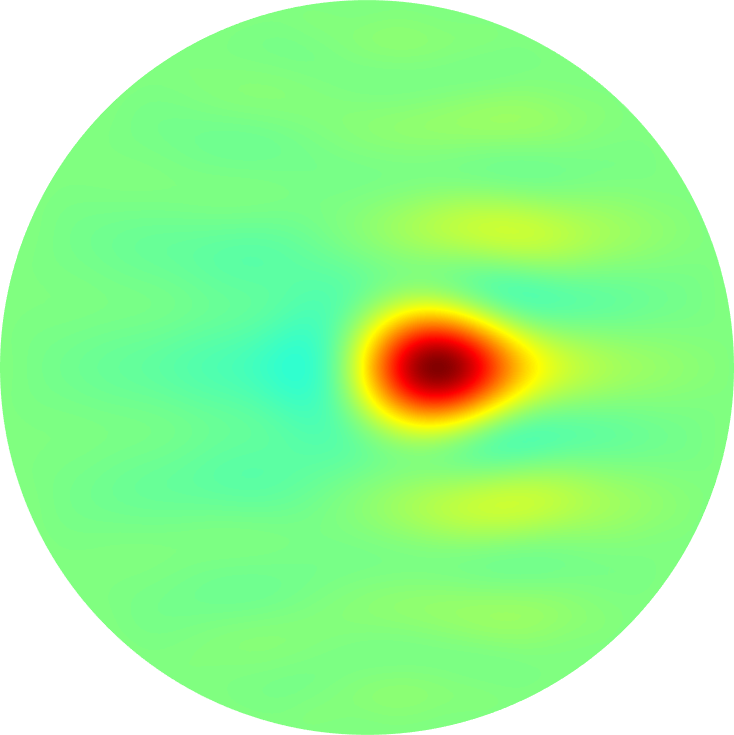}
 \end{tabular}
 \caption{Sensitivity functions for $n=17$, $q = 1$, i.e.
 the ``rows'' of $D\vQ[\vM(q)] D\vM[q]$. The other
 sensitivity functions can be obtained by rotations of integer multiples
 of $2\pi/17$.}
 \label{fig:dqdq}
\end{figure}
\begin{proposition}
\label{prop:null}
The Jacobian $D\vQ[\vM] : \real^{n\times n} \to
\real^{n(n-1)/2}$ of $\vQ[\vM]$ has rank $n(n-1)/2 - 1$. Its left null space is 
spanned by  the
 vector 
 \[
  \vz = \widetilde{\vq}_{\avg} \odot
  \frac{\vgamma(\vM)}{\vgamma_0}
 ~\mbox{i.e.,  $D\vQ[\vM]^T \vz = \vzero$.} 
 \]
 Its right null space is spanned by $\vM$ i.e., 
  $D\vQ[\vM] \vM = \vzero$.
\end{proposition}
\begin{proof}
We first characterize the left and right null space of the linearization
of the mapping $\widetilde{\vq}(\vgamma)$ defined by the
discrete Liouville identity \eqref{eq:dlv:q} with discrete reference
conductivity $\vgamma_0 \equiv \vgamma(\vM(0))$ and $\vgamma_1 \equiv\vgamma$.
Straightforward but lengthy calculations give that for some
$\delta\vgamma \in \real^{n(n-1)/2}$,
\[
 D\widetilde{\vq}[\vgamma] \delta\vgamma =
 - \M{\frac{\vgamma_0}{\vgamma}}^{1/2} \odot 
 \Mb{ \widetilde{\vL}_{\vgamma_0} \Mb{\frac{\delta\vgamma}{(\vgamma \odot
 \vgamma_0)^{1/2}}}} - \frac{\delta\vgamma}{\vgamma} \odot \widetilde{\vq}(\vgamma).
\]
In particular, if we let $\delta \vgamma = \vgamma$ we obtain from  \eqref{eq:dlv:q}
that 
\[
D\widetilde{\vq}[\vgamma]\vgamma = \vzero,
\]
so $\vgamma$ is a right null vector of the Jacobian. This  may 
be seen directly from \eqref{eq:dlv:q}, by realizing that the function $\vgamma \to \widetilde \vq(\vgamma)$ is homogeneous of degree zero
\footnote{
 The multiplication of  
$\vgamma_1 \equiv \vgamma$ in \eqref{eq:dlv:q}  by any  constant $\alpha > 0$ gives $\widetilde \vq(\alpha \vgamma) = \widetilde \vq(\vgamma)$, because the constants cancel out.}.

A similar calculation
gives that $\vgamma/\vgamma_0$ is a left null-vector for
$D\widetilde{\vq}[\vgamma]$,
\[
D\widetilde{\vq}[\vgamma]^T \M{\frac{\vgamma}{\vgamma_0}} = \vzero.
\]
Going back to the definition of $\vQ(\vM)$, we recall that for critical circular
planar graphs, the Jacobian $D\vgamma[\vM]$ is invertible
\cite[\S 12]{Curtis:1998:CPG}. Thus, step (1) in the definition of $\vQ$ gives an 
invertible linearization map. Step (3) is also invertible because it involves the multiplication  with a diagonal matrix with
nonzero entries, so the left and right
null-vectors of $D\vQ[\vM]$ can be found from those of
$D\widetilde{\vq}[\vgamma(\vM)]$.

The left null vector follows from $\vz/\widetilde{\vq}_{\avg} =
\vgamma(\vM)/\vgamma_0$ being in
the left nullspace of $D\widetilde{\vq}[\vgamma(\vM)]$, as shown above, where the division is 
understood componentwise.  
For the right null vector, consider a  scalar
$h > -1$. Then by the homogeneity of order $1$ between the DtN map and the conductors in the network, we must have
\[
 \widetilde{\vq}(\vgamma((1+h)\vM)) =
 \widetilde{\vq}((1+h)\vgamma(\vM)) =
 \widetilde{\vq}(\vgamma(\vM)).
\]
This and definition \eqref{eq:defQM} imply that $D\vQ[\vM] \vM = \vzero$. 
\end{proof}

\subsection{The initial guess}
\label{sec:init}
The localization of the sensitivity functions displayed in Figure
\ref{fig:dqdq} allows us to view $\vQ(\vM(q_{\true}))$ as
an approximation of $q_{\true}$ at the maxima of the sensitivity functions.
These maxima define the nodes of the \emph{sensitivity grids}
used in \cite{Borcea:2010:PRN} for the inverse conductivity problem.  The sensitivity grids depend weakly 
on the Schr\"odinger potential, as illustrated numerically in
figure~\ref{fig:qgrid} for two constant Schr\"odinger potentials. Thus, we can compute them ahead of time.

To obtain a good initial guess $q_0$ of the iteration \eqref{eq:gn}, we
linearly interpolate the values of $\vQ(\vM(q_{\true}))$ on a Delaunay
triangulation of the sensitivity grid nodes. Some of these initial
guesses are illustrated in figures~\ref{fig:gnsm} and \ref{fig:gnpc},
with the same color scales as the corresponding $q_{\true}$. We
note that $q_0$ is already close to the actual
Schr\"odinger potential, capturing not only its geometrical features but also its
magnitudes.  Moreover the computation of $q_0$ is inexpensive,
because all the operations involved in the calculation of $\vQ(\vM)$, with
given $\vM$, are carried on a relatively small network.

\begin{figure}
 \begin{tabular}{c@{\hspace{1em}}c}
 \includegraphics[width=0.4\textwidth]{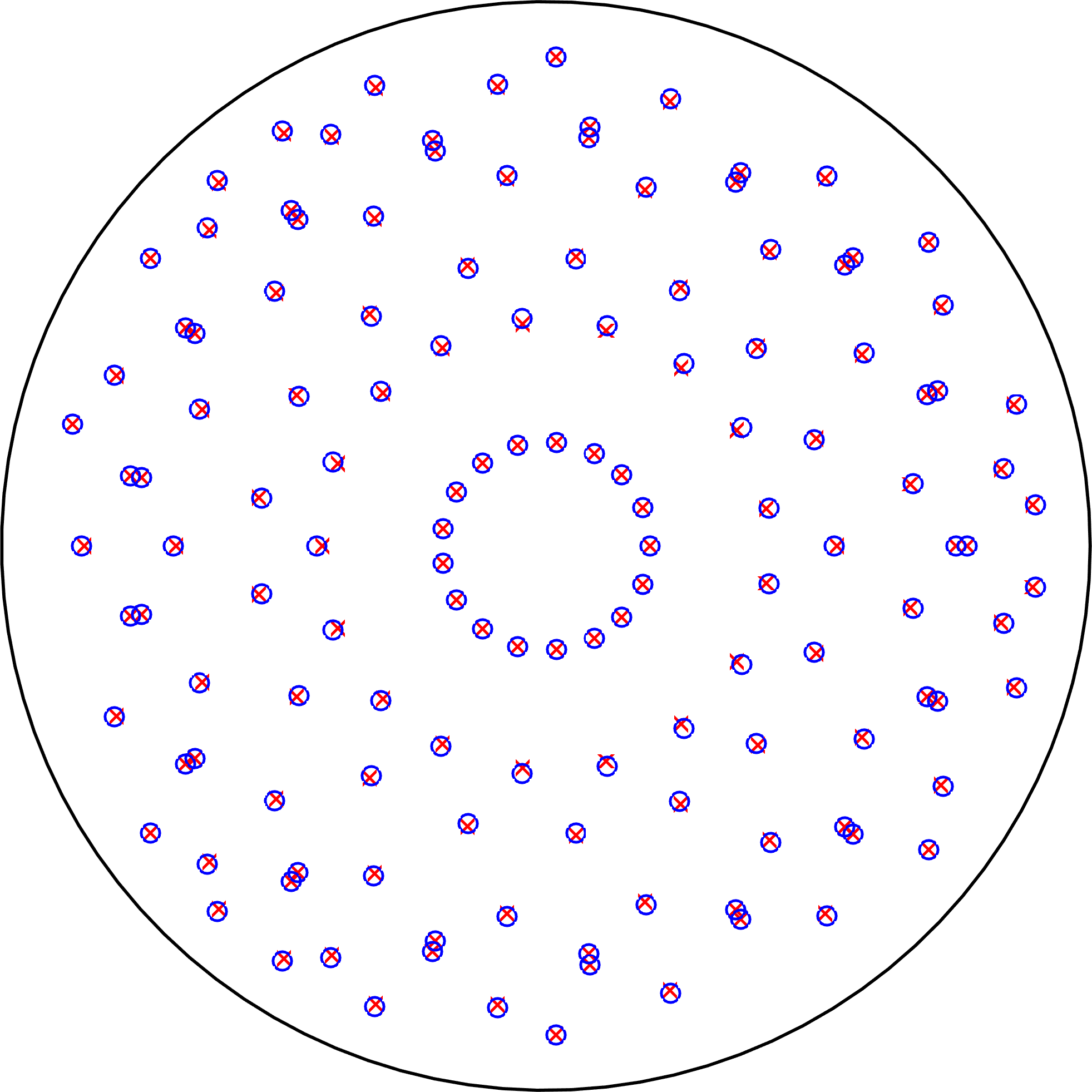} &
 \includegraphics[width=0.4\textwidth]{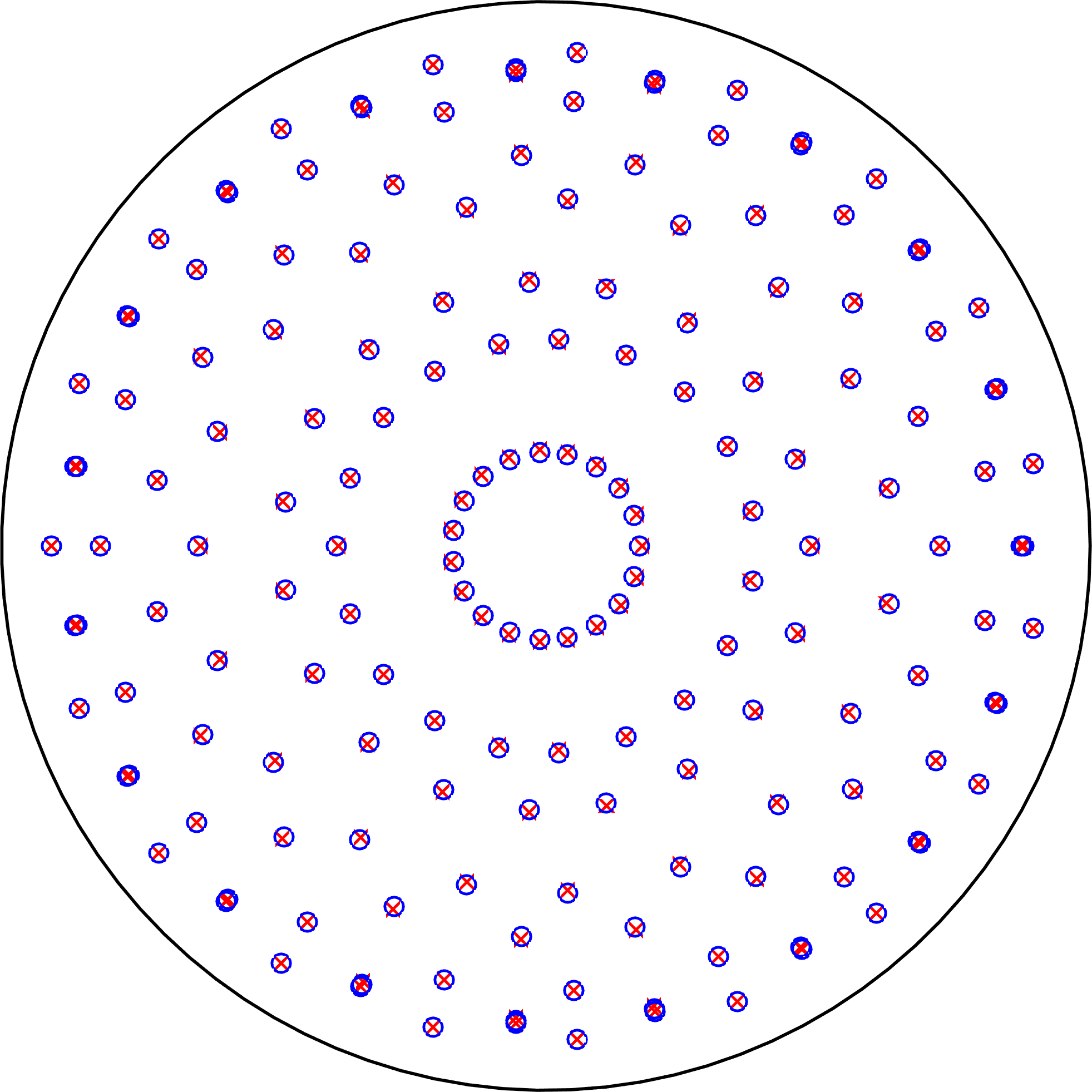}\\
 $n=17$ & $n=19$
 \end{tabular}
 \caption{Sensitivity grids. The ``x'' are for $q=1$ and the
 ``$\circ$'' for $q=3$.}
 \label{fig:qgrid}
\end{figure}

\section{Numerical results}
\label{sec:num}

We show here numerical reconstructions obtained from data  
$\vM(q_{\true})$ approximated by solving  (\ref{eq:sq})
numerically with a finite volume method on a fine grid, and then using 
the lumped measurement formulas (\ref{eq:smeasmat}).
With the same finite volume method, fine grid and measurement formulas,
we also compute $\vM(0)$  so that we can obtain 
$\vgamma_0 \equiv \vgamma(\vM(0))$ needed for the computation 
of $\vQ[\vM]$. This is to avoid discretization discrepancies in the calculation
of the preconditioner mapping. In practice, one may use any fine grid that
gives a good approximation of $\vM(0)$. 

\subsection{Estimating the average Schr\"odinger potential}
\label{sec:avgq}
To compute $\vQ[\vM]$ we need to estimate the average $q_{\avg}$ 
of the Schr\"odinger potential from the data $\vM(q_{\true})$. We do this 
with a direct search. Explicitly, given some trial values $q_1,\ldots,q_m$ for the 
average, 
we estimate it as
\[
 q_{\avg} = \argmin_{q \in \{q_1,\ldots,q_m\}} \| \vM(q) -
 \vM(q_{\true}) \|_F,
\]
where $\| \cdot \|_F$ is the Frobenius norm. To avoid the inverse crime, in this estimation
we compute $\{\vM(q_j)\}_{1 \le j \le m}$ on a different fine grid than the one used for computing the
synthetic data $\vM(q_{\true})$. 

\subsection{Calculating the non-linear preconditioner mapping}
The computation of $\vQ[\vM]$ involves 
$\vgamma_{\avg}$, the discrete conductivity corresponding to the 
constant Schr\"odinger potential $q \equiv q_{\avg}$. We obtain it from 
$\vM(q_{\avg})$,  computed  as described above  for $\vM(q_{\true})$ and $\vM(0)$. The 
fine grid in the computation of $\vM(q_{\avg})$ is  the same as that in the calculation of  $\vM(0)$, to avoid discretization 
discrepancies.

\subsection{The Gauss-Newton iterations}
\label{sect:GNiter}
The reconstructions obtained with Gauss-Newton iteration 
(\ref{eq:gn}) for a smooth and a piecewise constant potential are shown 
in figures~\ref{fig:gnsm} and
\ref{fig:gnpc}, respectively. We include only the initial guess $q_0$
(which is obtained as explained in \ref{sec:init}) and the first iterate
$q_1$. Subsequent iterates $q_k$, $k\geq 2$ are not
shown because they  are indistinguishable from $q_1$.
Thus, for all practical purposes the Gauss-Newton iteration converges
in one iteration. We believe this is because the initial guess $q_0$ is
already close to $q_{\true}$, and the Jacobian of $\vQ \circ \vM$
is in some sense close to the identity, as explained in
\S\ref{sec:sensitivity}. We observe that $q_1$ seems to be a better
reconstruction than $q_0$. For example in figure~\ref{fig:gnpc}, $q_0$
has artifacts close to the center due to the Delaunay triangulation used
for the linear interpolation. These artifacts are diminished in $q_1$.

We include a typical convergence history in figure~\ref{fig:iter}. Since
the Gauss-Newton iterations  are designed to minimize the
residual \eqref{eq:pols}, we expect that the norm of the {\em
preconditioned residual} $\|\vQ(\vM(q_k)) -
\vQ(\vM(q_{\true}))\|_2^2$ decreases with $k$. This is true for the
first iteration, but then the residual stagnates. To explain this,
consider a vector $\vz_k \neq \vzero$ spanning the null space of the Jacobian of
$\vQ(\vM(q))$ evaluated at $q = q_k$.  An explicit formula for such a
vector is given in proposition~\ref{prop:null}. Clearly the Gauss-Newton
update $q_{k+1} - q_k$ defined in \eqref{eq:gn} is independent of the
component of the preconditioned residual $\vQ(\vM(q_k)) -
\vQ(\vM(q_{\true}))$ in the direction $\vz_k$. In other words, the update
\[
q_{k+1} -  q_k 
= - D\vM[q_k]^\dagger D\vQ[\vM(q_k)]^\dagger(
\vQ(\vM(q_k)) - \vQ(\vM(q_{\true})) + \alpha \vz_k),
\]
is independent of $\alpha$. We believe the stagnation of the
preconditioned residual is because the component of the preconditioned
residual along $\vz_k$ is not zero, and  the component orthogonal to
$\vz_k$ is still reduced by the iterations. To test this hypothesis, we 
include in figure~\ref{fig:iter} the norm of the {\em projected
preconditioned residual} $ \|\vP_k(\vQ(\vM(q_k)) -
\vQ(\vM(q_{\true})))\|_2^2$, where $\vP_k = \vI -
\vz_k\vz_k^T/(\vz_k^T\vz_k)$ is the projector on the subspace  orthogonal to
$z_k$. We observe that
this quantity does decrease with $k$, at least until
reaching machine precision (at $k=2$). For reference we
also include  the norm of the {\em unpreconditioned residual}, i.e. $\|
\vM(q_k) - \vM(q_{\true})\|_F^2$. This quantity remains  basically unchanged with the
iterations.

\begin{figure}
 \begin{center}
  \includegraphics[page=1,width=\textwidth]{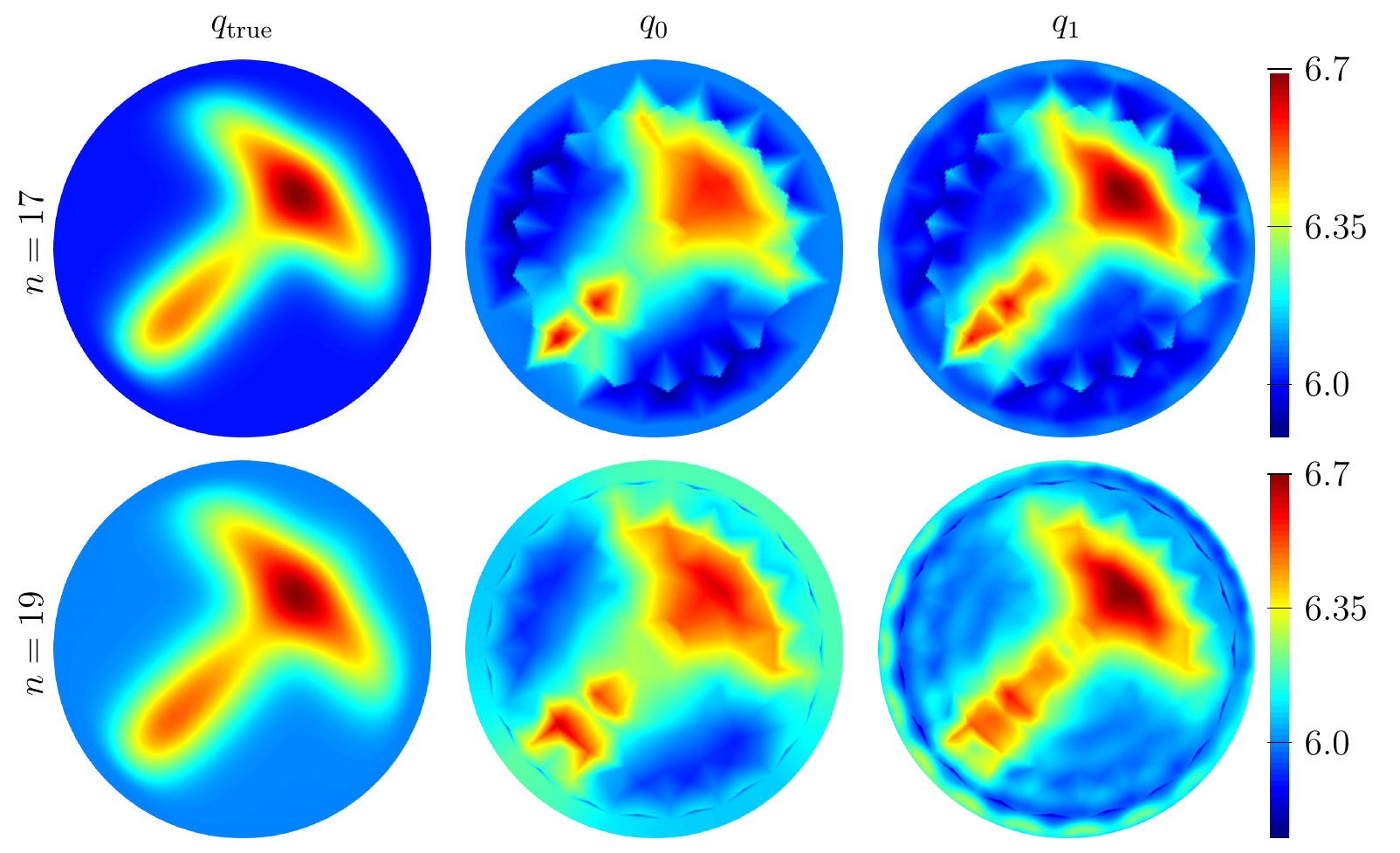}
 \end{center}
 \caption{Gauss-Newton iterates for smooth $q$ (sensitivity
 grid)}
 \label{fig:gnsm}
\end{figure}

\begin{figure}
 \begin{center}
  \includegraphics[page=2,width=\textwidth]{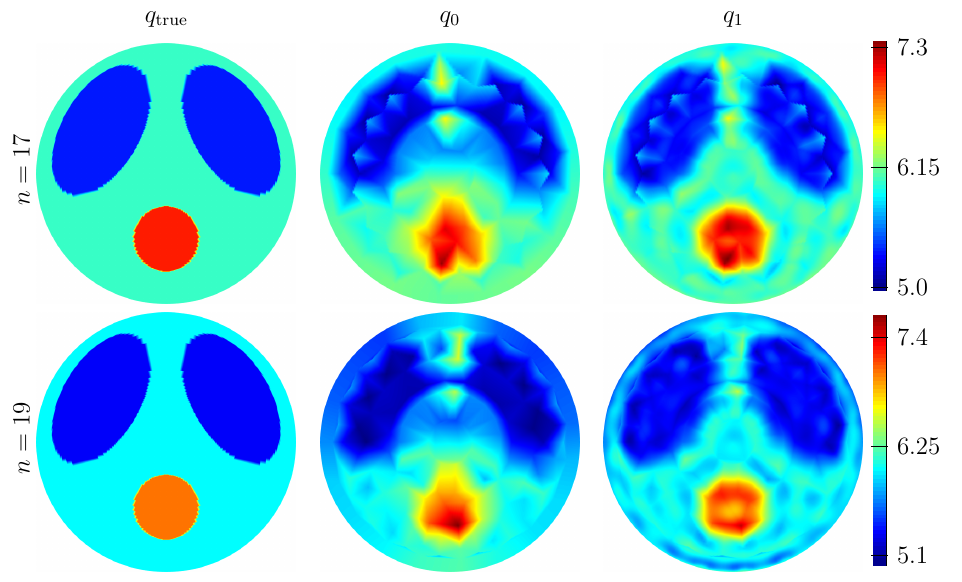}
 \end{center}
 \caption{Gauss-Newton iterates for piecewise constant $q$ (sensitivity
 grid)}
 \label{fig:gnpc}
\end{figure}

\begin{figure}
 \begin{center}
  \includegraphics[width=0.5\textwidth]{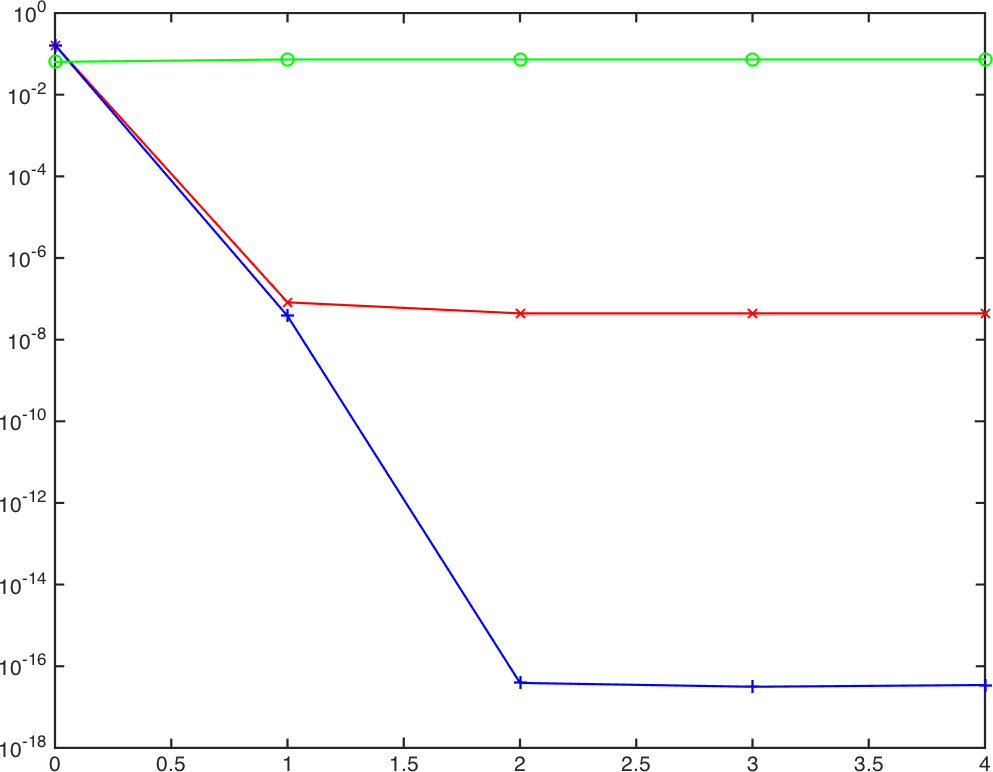}

  Iteration
 \end{center}
 \caption{A typical convergence history for the preconditioned
 Gauss-Newton method. We show convergence in terms of the
 unpreconditioned residual (green), the preconditioned residual (red)
 and the projected preconditioned residual (blue).}  \label{fig:iter}
\end{figure}

\section{Summary}
\label{sec:discussion}

We introduced and analyzed a novel inversion algorithm for determining  absorptive Schr\"odinger
potentials from measurements of the Dirichlet to Neumann map.  The method can be 
viewed in the context of parametric model reduction, with reduced models having the physical meaning
of resistor networks. Networks have been used successfully for the inverse conductivity problem. 
Here we show how to use them for Schr\"odinger's equation. The connection is made by a new 
discrete generalized Liouville identity, defined on the line graph of the network. The set 
of nodes of the line graph, where  the potential is discretized, is isomorphic to the set of edges 
of the network, where the conductivities are defined. 

We use the discrete Liouville identity to obtain 
an approximate inverse of the forward map, and formulate   
a preconditioned Gauss-Newton 
iteration for reconstructing the continuum Schr\"odinger potential. The method is 
superior to traditional output least squares because it converges quickly, in one step, it is computationally inexpensive and gives 
good quantitative results.

\section*{Acknowledgements}
LB acknowledges partial support from ONR Grant N00014-14-1-0077 and NSF grant DMS-1510429. The
work of FGV was partially supported by the National Science Foundation
grant DMS-1411577.

\appendix
\section{Choosing the size of the network}
\label{app:lump}

If the measurements of $\Lambda_{1,q}$ are tainted with noise,
we regularize the inversion by reducing  the size of the network i.e.,
the number $n$ of its boundary nodes, at the expense of resolution. Suppose that we start with a
large number $N$ of such points. Due to the instability of the inverse
problem and the noise, we may get negative resistors at step (1)
in the calculation of the mapping $\vQ$. Thus, we 
reduce the number of boundary nodes until we obtain positive resistors. We call the 
largest such number $n$ and note that it  
depends on the noise level. 

The reduction of the data during this process  is done by lumping nearby measurements.
 Instead of working with the $N$ lumping functions 
$\phi_1,\ldots,\phi_N$, we use the  fewer $\psi_1,\ldots,\psi_n$ 
defined by
\[
 \psi_i = \sum_{j\in S_i} \alpha_{i,j} \phi_j,~\mbox{for}~i=1,\ldots,n.
\]
Here the $S_i \neq \emptyset$ are $n$ disjoint subsets of
$\{1,\ldots,N\}$ corresponding to consecutive $\phi_j$ in
$\partial\Omega$ and $\alpha_{i,j}$ are positive weights with
$\sum_{j \in S_i} \alpha_{i,j} = 1$ so that we also have
$\int_{\partial\Omega}\psi_i = 1$. With the lumping we get the
off-diagonal elements of a smaller $n\times n$ data matrix $\vM'$ from
$\vM$. The diagonal elements of $\vM'$ are obtained from the
off-diagonal ones by enforcing conservation of currents,  so we can
guarantee \cite{Borcea:2008:EIT} (at least in the noiseless case) that
there is a unique resistor network with graph $C\big(n(n-1)/2,n\big)$ 
and DtN map $\vM'$.

The lumping has two regularizing effects. The first  is that
by summing more measurements of $\Lambda_{1,q}$ we get noise cancellation. The second is that the estimation of the resistors in a
network from its DtN map is more stable if the network is smaller.

\bibliographystyle{abbrvnat}
\bibliography{schroe}
\end{document}